%% file: classe.tex
\documentclass[12pt,a4paper,leqno]{article}
\usepackage[utf8x]{inputenc}
\usepackage[francais]{babel}
\usepackage[T1]{fontenc}
\usepackage{latexsym}
\usepackage{amsmath}
\usepackage{amssymb, amscd, amsthm, mathrsfs}
\usepackage[all]{xy}
\usepackage[dvips]{graphicx}
\usepackage{makeidx}
\newtheorem{theo}{{Th\'eor\`eme}}[section]
\newtheorem{coro}[theo]{{Corollaire}}
\newtheorem{lemma}[theo]{{Lemme}}
\newtheorem{prop}[theo]{Proposition}

\theoremstyle{remark}
\newtheorem{remark}[theo]{\textbf{Remarque}}

\theoremstyle{definition}
\newtheorem{defn}[theo]{D\'efinition}

\newtheorem{recette}[theo]{\textbf{Recette}}

\newcommand{\ra}{\rightarrow}

\newcommand{\ol}{\overline}
\newcommand{\immouv}[1][r]
   {\ar@{}[#1] |*[o][F]{\hbox{%
         \vrule width 1.5mm height 0pt depth 0pt%
         \vrule width 0pt height .75mm depth .75mm%
         }}
         \ar@{^{(}->}[#1]}

\newcommand{\cC}{\mathcal{C}}

\newcommand{\cG}{\mathcal{G}}
\newcommand{\cH}{\mathcal{H}}

\newcommand{\cK}{\mathcal{K}}

\newcommand{\cM}{\mathcal{M}}

\newcommand{\cO}{\mathcal{O}}
\newcommand{\cP}{\mathcal{P}}

\newcommand{\cU}{\mathcal{U}}

\newcommand{\A}{\mathbb A}
\newcommand{\C}{\mathbb C}
\newcommand{\B}{\mathbb B}

\newcommand{\N}{\mathbb N}

\newcommand{\Q}{\mathbb Q}
\newcommand{\R}{\mathbb R}

\newcommand{\Z}{\mathbb Z}

\newcommand{\bM}{\mathbf M}
\newcommand{\bR}{\mathbf R}

\newcommand{\fD}{\mathfrak D}

\newcommand{\fK}{\mathfrak K}

\newcommand{\rH}{\mathrm H}

\newcommand{\m}{\mathrm m}
\newcommand{\rP}{\mathrm P}

\renewcommand{\ker}{\mathrm {Ker} }

\DeclareMathOperator{\coker}{\mathrm Coker }

\renewcommand{\Re}{\mathrm Re}

\DeclareMathOperator{\GL}{\mathrm GL}
\DeclareMathOperator{\SL}{\mathrm SL}
\DeclareMathOperator{\LC}{\mathrm LC}

\DeclareMathOperator{\Gal}{\mathrm Gal}

\DeclareMathOperator{\Sym}{\mathrm{Sym}}

\DeclareMathOperator{\plim}{\varprojlim}

\newcommand{\alg}{\mathrm{alg}}

\newcommand{\con}{\mathrm{cong}}

\newcommand{\cycl}{\mathrm{cycl}}
\newcommand{\dR}{\mathrm{dR}}

\newcommand{\Inf}{\mathrm{inf}}

\newcommand{\res}{\mathrm{res}}

\newcommand{\z}{\zeta}

\begin{document}

\title{Une loi de réciprocité explicite pour le polylogarithme elliptique}
\author{Francesco Lemma, Shanwen Wang\footnote{Le deuxième auteur est financé par le projet ANR ArSHiFo.}}
\date{}
\maketitle

\input{notation.tex}

\input{systeme.tex}

\input{Siegel.tex}

\input{loi.tex}

--------------------------------------

\noindent Francesco Lemma,\\
Université Paris 7 - Denis Diderot\\
Institut de Mathématiques de Jussieu - Paris Rive Gauche\\
UMR7586\\
Bâtiment Sophie Germain\\
Case 7012\\
75205 PARIS Cedex 13\\
France\\
Courrier électronique : lemma@math.univ-paris-diderot.fr\\
--------------------------------------\\
\noindent Shanwen Wang\\
Institut de mathématiques de Jussieu, 4 place Jussieu, 75005 Paris, France \\ Courrier électronique : wetiron1984@gmail.com 
\end{document}

%% file: notation.tex
\begin{abstract}
On démontre une compatibilité entre la réalisation $p$-adique et la réalisation de de Rham des sections de torsion du profaisceau polylogarithme elliptique. La preuve utilise une variante pour $\rH^1$ de la loi de réciprocité explicite de Kato pour le $\rH^2$ des courbes modulaires.\end{abstract}
\def\abstractname{Abstract}
\begin{abstract} We prove a compatibility between the $p$-adic realization and the de Rham realization of the torsion sections of the elliptic polylogarithm prosheaf.  The proof uses a new explicit reciprocity law for $\rH^1$, which is a variant of Kato's explicit reciprocity law for $\rH^2$ of the modular curves.  
\end{abstract}
\tableofcontents

\section{Introduction et Notations }
\subsection{Introduction}
Ce travail est motiv\'e par l'\'etude des valeurs sp\'eciales des fonctions $L$ motiviques. Des conjectures tr\`es g\'en\'erales de Beilinson et Bloch-Kato donnent une interpr\'etation cohomologique de ces nombres complexes en termes de cohomologie motivique et de r\'egulateurs. Comme la cohomologie motivique est au jour d'aujourd'hui incalculable en g\'en\'eral, la strat\'egie utilis\'ee dans tous les cas o\`u l'on a pu r\'esoudre, ou "presque", ces conjectures est de construire explicitement des classes de cohomologie motivique particuli\`eres dont le r\'egulateur est calculable, puis d'utiliser des techniques automorphes.

 Le symbole d'Eisenstein, construction due \`a Beilinson, fournit des classes de cohomologie motivique non-triviales sur les produits fibr\'es de la courbe elliptique universelle sur une courbe modulaire. Ces classes interviennent de mani\`ere plus ou moins directe dans beaucoup de d\'emonstrations, ou d'approches, de cas particuliers des conjectures de Beilinson et Bloch-Kato (formes modulaires elliptiques, caract\`eres de Dirichlet, formes modulaires de Hilbert sur un corps quadratique r\'eel, formes modulaires de Siegel de genre $2$...). Soyons plus pr\'ecis: soit $K \subset \mathrm{GL}_2(\mathbb{A}_f)$ un sous-groupe ouvert compact net et soit $Y_K$ le courbe modulaire de niveau $K$. On a la courbe elliptique universelle $\pi: E_K \longrightarrow Y_K$. Soit $t$ une section de torsion de $\pi$. Pour tout $n$ entier positif ou nul, la classe d'Eisenstein de poids $n+2$ associ\'ee \`a $t$ est un \'el\'ement
$$
\mathrm{Eis}_t^{n+2} \in H^1_\mathcal{M}(Y_K, \mathrm{Sym}^n \mathcal{H}(1))
$$
o\`u $\mathcal{H}$ d\'esigne le faisceau motivique $\underline{\mathrm{Hom}}(R^1\pi_* \mathbb{Q}(0), \mathbb{Q}(0))$ et o\`u $H^1_\mathcal{M}(Y_K, \mathrm{Sym}^n \mathcal{H}(1))$ d\'esigne la cohomologie motivique de $Y_K$ \`a coefficients dans $\mathrm{Sym}^n \mathcal{H}(1)$. Notons d\`es \`a pr\'esent que pour $K$ le sous-groupe principal de congruence de niveau $N$, et pour $t$ bien choisie, la classe $\mathrm{Eis}_t^0$ est l'image de l'unit\'e de Siegel sous l'application de Kummer intervenant dans la construcation du syst\`eme d'Euler de Kato (cf. \cite{PC1},\cite{KK},\cite{Wang}...). Ceci sera rendu pr\'ecis dans le pr\'esent travail. Comme nous l'avons mentionn\'e plus haut, pour obtenir des applications aux valeurs sp\'eciales de fonctions $L$, il est crucial de d\'ecrire aussi explicitement que possible l'image par les r\'egulateurs (de Beilinson, Betti, $l$-adique, syntomique, de Rham...) de ces classes de cohomologie motivique. La description de l'image des classes d'Eisenstein par le r\'egulateur de Beilinson a \'et\'e donn\'ee par Beilinson lui-m\^eme, et la r\'ealisation de Betti s'en d\'eduit formellement. La r\'ealisation $l$-adique a \'et\'e d\'etermin\'ee par Kings (\cite{GK} Thm. 4.2.9), la r\'ealisation syntomique et de de Rham a \'et\'e d\'etermin\'ee par Bannai-Kings (\cite{BK} Prop. 3.8). 

Le but de cet article est de donner une d\'emonstration de la description des classes d'Eisenstein en cohomologie de de Rham enti\`erement diff\'erente de celle de Bannai-Kings et de placer ce r\'esultat dans un cadre plus g\'en\'eral. En effet, nous utilisons le langage de Colmez des distributions alg\'ebriques pour encoder les relations de distribution satisfaites par les symboles d'Eisenstein de niveaux diff\'erents, ainsi qu'une nouvelle loi de r\'eciprocit\'e explicite (voir le théorème \ref{principal} pour un énoncé précis). De manière un peu plus précise, on définit deux distributions algébriques\footnote{G. Kings \cite{GK1} donne une perspective similaire sur les classe d'Eisenstein $p$-adiques en utilisant le language des faisceaux de modules d'Iwasawa.} (voir \S 2 et \S 3 respectivement) sur un espace localement profini $X^{(p)}$ à valeurs dans l'algèbre des formes modulaires et dans la cohomologie étale de courbe modulaire à coefficient dans $W_k(1)$ respectivement, où $W_k=\Sym^{k-2}V_p$ avec $V_p=\Q_pe_1\oplus \Q_p e_2$ la représentation standard de $\GL_2(\Z_p)$, et on montre qu'elles sont reliées par l'application $\exp^*$, où $\exp^*$ est une variante de l'application exponentielle duale de Kato, dont la formulation est très semblable à celle de Kato revisitée par Colmez \cite{PC1} (cf. aussi \cite{Wang}) mais pour $\rH^1$ au lieu de $\rH^2$.

 \subsection{Notations}
On note $\overline\Q$ la cl\^oture alg\'ebrique de $\Q$ dans $\C$ et on fixe, pour tout nombre premier~$p$, une cl\^oture
alg\'ebrique $\overline\Q_p$ de $\Q_p$, ainsi qu'un plongement de
$\overline\Q$ dans $\overline\Q_p$.

Si $N\in\N$, on note $\zeta_N$ la racine $N$-i\`eme
$e^{2i\pi/N}\in\overline\Q$ de l'unit\'e. On note
$\Q^{\rm cycl}$ l'extension cyclotomique de $\Q$,
r\'eunion des $\Q(\zeta_N)$, pour $N\geq 1$, ainsi que $\Q^{\rm cycl}_p$ l'extension cyclotomique de $\Q_p$, r\'eunion de $\Q_p(\z_N)$, pour $N\geq 1$.
 
 \subsubsection*{Objets ad\'eliques}
 Soient $\cP$ l'ensemble des premiers de $\Z$ et $\hat{\Z}$ le compl\'et\'e profini de $\Z$, alors
$\hat{\Z}=\prod_{p\in\cP}\Z_p$. Soit $\A_f=\Q\otimes\hat{\Z}$ l'anneau
des ad\`eles finies de $\Q$. Si
$x\in\A_f$, on note $x_p$ (resp. $x^{]p[}$) la
composante de $x$ en $p$ (resp. en dehors de $p$). Notons
$\hat{\Z}^{]p[}=\prod_{l\neq p}\Z_l$. On a donc
$\hat{\Z}=\Z_p\times\hat{\Z}^{]p[}$. Cela induit les d\'ecompositions
suivantes: pour tout  $d\geq 1$,
\[
\bM_d(\A_f)=\bM_d(\Q_p)\times\bM_d(\Q\otimes\hat{\Z}^{]p[})
\text{ et }
\GL_d(\A_f)=\GL_d(\Q_p)\times\GL_d(\Q\otimes\hat{\Z}^{]p[}).\]
\subsubsection*{Actions de groupes}
Soient $X$ un espace topologique localement profini, $V$ un
$\Z$-module. On note $\LC_c(X,V)$ le module des fonctions localement
constantes sur  $X$ \`a valeurs dans $V$ dont le support est compact
dans $X$. On note $\fD_{\alg}(X,V)$ l'ensemble des distributions alg\'ebriques sur
$X$ \`a valeurs dans $V$, c'est-\`a-dire, des applications
$\Z$-lin\'eaires de $\LC_c(X,\Z)$ \`a valeurs dans $V$. On note $\int_X\phi\mu$ la valeur de $\mu$
sur $\phi$, o\`u $\mu\in\fD_{\alg}(X,V)$ et $\phi\in \LC_c(X,\Z)$.

Soit $G$ un groupe localement profini, agissant contin\^ument \`a
droite sur $X$ et $V$. On munit $\LC_c(X,\Z)$ et
$\fD_{\alg}(X,V)$ d'actions de $G$ \`a droite comme suit: si $g\in G, x\in X,\phi\in\LC_c(X,\Z), \mu\in\fD_{\alg}(X,V),$ alors
\begin{equation}\label{actiondis} (\phi*g)(x)=\phi(x*g^{-1}) \text{ et } \int_{X}\phi(\mu*g)=\bigl(\int_{X}(\phi*g^{-1})\mu\bigr)*g.
\end{equation}

\subsubsection*{Formes modulaires}
Soient $A$ un sous-anneau de $\C$ et $\Gamma$ un sous-groupe
d'indice fini de $\SL_2(\Z)$. On note $\cM_k(\Gamma,\C)$ le
$\C$-espace vectoriel des formes modulaires de poids $k$ pour
$\Gamma$. On note aussi $\cM_{k}(\Gamma,A)$ le sous $A$-module de
$\cM_k(\Gamma,\C)$ des formes modulaires dont le $q$-d\'eveloppement
est \`a coefficients dans $A$. On pose
$\cM(\Gamma,A)=\oplus_{k=0}^{+\infty}\cM_k(\Gamma,A)$. Et on note
$\cM_k(A)$ (resp. $\cM(A)$) la r\'eunion des $\cM_k(\Gamma,A)$
(resp. $\cM(\Gamma,A)$), o\`u $\Gamma$ d\'ecrit tous les
sous-groupes d'indice fini de $\SL_2(\Z)$. On note $L^\con$ l'ensemble des sous-groupes de congruence. On d\'efinit de m\^eme:
\[\cM^{\con}_k(A)=\bigcup\limits_{\substack{\Gamma \in L^\con}}\cM_k(\Gamma,A)\text{ et }
\cM^{\con}(A)=\oplus_{k=0}^{+\infty}\cM_k^{\con}(A).\]



Soit $K$ un sous-corps de $\C$ et soit $\ol{K}$ la cl\^oture alg\'ebrique de $K$ dans $\C$. On note $\Pi_K$ le groupe des
automorphismes de $K$-algèbres graduées $\cM(\ol{K})$ sur $\cM(\SL_2(\Z),K)$; c'est un
groupe profini. Si $K$ est alg\'ebriquement clos et si $\Gamma$ est un sous-groupe distingu\'e d'indice fini de $\SL_2(\Z)$, alors le groupe des automorphismes de $\cM(\Gamma,K)$ sur $\cM(\SL_2(\Z),K)$ est $\SL_2(\Z)/\Gamma$. On en déduit que $\Pi_K=\widehat{\SL_2(\Z)}$, o\`u $\widehat{\SL_2(\Z)}$ est le compl\'et\'e profini de $\SL_2(\Z)$. Dans le cas g\'en\'eral, on dispose d'une suite exacte:
\[1\ra\Pi_{\ol{K}}\ra\Pi_K\ra\cG_K\ra1,\]
qui admet une section $\cG_K\ra\Pi_K$ naturelle,  en faisant agir $\cG_K$ sur les
coefficients du $q$-d\'eveloppement des formes modulaires.
Le groupe des automorphimes d'algèbres de $\cM^{\con}(\Q^{\cycl})$ sur $\cM(\SL_2(\Z),\Q^{\cycl})$ est le groupe profini $\SL_2(\hat{\Z})$, compl\'et\'e de $\SL_2(\Z)$ par rapport aux sous-groupes de congruence. D'autre part, quel que soit $f\in\cM^{\con}(\Q^{\cycl})$, le groupe $\cG_{\Q}$ agit sur les coefficients du $q$-d\'eveloppement de $f$ \`a travers son quotient $\Gal(\Q^{\cycl}/\Q)$ qui est isomorphe \`a $\hat{\Z}^{*}$ par le caract\`ere cyclotomique. La sous-alg\`ebre $\cM^{\con}(\Q^{\cycl})$ est stable par $\Pi_{\Q}$ qui agit \`a travers $\GL_2(\hat{\Z})$ et on a le diagramme commutatif de groupes suivant (cf. par exemple \cite[théorème 2.2]{Wang}):

\[\xymatrix{
1\ar[r]&\Pi_{\bar{\Q}}\ar[r]\ar[d]&\Pi_{\Q}\ar[r]\ar[d]&\cG_{\Q}\ar[r]\ar[d]^{\chi_{\cycl}}\ar@{.>}@/^/[l]^{\iota_\Q}&1\\
1\ar[r]&\SL_{2}(\hat{\Z})\ar[r]&\GL_2(\hat{\Z})\ar[r]^{\det}&\hat{\Z}^{*}\ar[r]\ar@{.>}@/^/[l]^{\iota}&1                     },\]
où la section $\iota_\Q$ de $\cG_{\Q}$ dans $\Pi_{\Q}$ d\'ecrite plus haut envoie $u\in \hat{\Z}^{*}$ sur la matrice $(\begin{smallmatrix}1&0\\0&u\end{smallmatrix})\in \GL_2(\hat{\Z})$.

\subsection{Remerciements}

Nous sommes heureux de remercier Pierre Colmez pour ses remarques. La plupart de ce travail a été effectué pendant le séjour du deuxième auteur à l'université de Padoue. Il a bénéficié de discussions intéressantes avec Matteo Longo et René Scheider. Une grande partie de cet article a été rédigée pendant ses séjours à l'IMJ et à l'IHES. Il souhaite remercier ces institutions pour lui avoir fourni d'excellentes conditions de travail.

%% file: systeme.tex
\section{La distribution
$z_{u,\mathrm{Eis,dR}}(k)$}\label{section1}

\subsection{S\'eries d'Eisenstein-Kronecker }\label{EK}
Les r\'esultats de ce paragraphe peuvent se trouver dans le livre de Weil $\cite{Weil}$.

\begin{defn}Si $(\tau,z)\in\cH\times\C$, on pose $q=e^{2i\pi\tau}$ et
$q_z=e^{2i\pi z}$. On introduit l'op\'erateur
$\partial_{z}:=\frac{1}{2i\pi}\frac{\partial}{\partial
z}=q_z\frac{\partial}{\partial q_z}$. On pose aussi  $e(a)=e^{2i\pi
a}$. Si $k\in\N$, $\tau\in\cH$, et $z,u\in\C$, on définit la s\'erie d'Eisenstein-Kronecker par
\[\rH_k(s,\tau,z,u)=\frac{\Gamma(s)}{(-2i\pi)^k}(\frac{\tau-\bar{\tau}}{2i\pi})^{s-k}\sideset{}{'}\sum_{\omega\in\Z+\Z\tau}\frac{\overline{\omega+z}^k}{|\omega+z|^{2s}}e(\frac{\omega\bar{u}-u\bar{\omega}}{\tau-\bar{\tau}}).\]
Elle converge pour $\mathrm{Re}(s)>1+\frac{k}{2}$, et poss\`ede un
prolongement
 m\'eromorphe \`a tout le plan complexe avec des p\^oles simples
en $s=1$ (si $k=0$ et $u\in\Z+\Z\tau$) et $s=0$ (si $k=0$
et $z\in\Z+\Z\tau$). Dans la formule ci-dessus $\sideset{}{'}\sum$
signifie (si $z\in\Z+\Z\tau$) que l'on supprime le terme
correspondant \`a $\omega=-z$. 
\end{defn}
Si $k\geq 1$, on d\'efinit les fonctions suivantes:
\begin{align*}
E_k(\tau,z)=\rH_k(k,\tau,z,0),\,  \, \, F_k(\tau,z)=\rH_k(k,\tau,0,z).
\end{align*}
Les fonctions $E_k(\tau,z)$ et $F_k(\tau,z)$ sont p\'eriodiques
en $z$ de p\'eriode $\Z\tau+\Z$. De plus on a:
\[E_{k+1}(\tau,z)=\partial_z E_k(\tau,z), \text{ si } k\in\N \text{ et } E_0(\tau,z)=\log|\theta(\tau,z)| \text{ si }z\notin \Z+\Z\tau,\]
o\`u $\theta(\tau,z)$ est donn\'ee par le produit infini:
\[\theta(\tau,z)=q^{1/12}(q_z^{1/2}-q_z^{-1/2})\sideset{}{}\prod_{n\geq 1}((1-q^n q_z)(1-q^nq_z^{-1})).\]

Soient $(\alpha,\beta)\in (\Q/\Z)^2$ et $(a,b)\in\Q^2$ qui a pour image
$(\alpha,\beta)$ dans $(\Q/\Z)^2$. Si $k=2$ et $(\alpha,\beta)\neq(0,0)$, ou si $k\geq 1$ et $k\neq 2$,  on d\'efinit des séries d'Eisenstein, éléments de $\cM^{\con}(\Q^{\cycl})$:
\[ E_{\alpha,\beta}^{(k)}=E_k(\tau,a\tau+b) \text{ et } F_{\alpha,\beta}^{(k)}=F_k(\tau,a\tau+b).\]
Si $k=2$ et $(\alpha,\beta)=(0,0)$, on d\'efinit\footnote{La s\'erie $H_2(s,\tau,0,0)$ converge pour $\Re(s)>2$, mais pas pour $s=2$.} $E_{0,0}^{(2)}=F_{0,0}^{(2)}:=\lim_{s\ra 2}H_2(s,\tau,0,0).$

\subsection{La distribution $z_{u,\mathrm{Eis,dR}}(k)$}\label{distributionEis}
 Soient $X=\A_f^2, G=\GL_2(\hat{\Z})$ et $V=\cM_k^{\con}(\Q^{\cycl})$. On définit une action de $G$ à droite sur $X$ par la multiplication de matrices:
 \[x*\gamma= (a,b)\gamma, \text{ si } x=(a,b) \text{ et }\gamma\in \GL_2(\hat{\Z}).\] 
L'action de $\Pi_\Q$ stabilise $V$ et se factorise à travers $\GL_2(\hat{\Z})$; l'action de $\SL_2(\hat{\Z})$ est l'action modulaire usuelle $|_k$ et celle de $(\begin{smallmatrix}1&0\\0&d\end{smallmatrix})$, si $d\in \hat{\Z}^*$, se fait via un rel\`evement $\sigma_d$ dans $\cG_{\Q}$ agissant sur les coefficients du $q$-d\'eveloppement.


La proposition suivante est une traduction des relations de distribution \cite[lemme 2.6]{Wang} pour les séries d'Eisenstein $F^{(k)}_{\alpha,\beta}$.
\begin{prop}\cite[théorème 2.13]{Wang}\label{eiskj} Si $k\geq 1$, il existe une distribution alg\'ebrique 
\[z_{\mathrm{Eis, dR}}(k)\in
\fD_{\alg}(\A_f^2,\cM_k^{\con}(\Q^{\cycl}))\] v\'erifiant:
quels que soient $r\in\Q^*$ et $(a,b)\in\Q^2$, on a
\[
\int_{(a+r\hat{\Z})\times(b+r\hat{\Z})}z_{\mathrm{Eis,dR}}(k)=r^{k-2}F^{(k)}_{r^{-1}a,r^{-1}b}.
\]
De plus, si $\gamma\in\GL_2(\hat{\Z})$, alors
$z_{\mathrm{Eis,dR}}(k)*\gamma=z_{\mathrm{Eis,dR}}(k).$
\end{prop}

Soit $\langle\cdot\rangle: \Z_p^{*}\ra \hat{\Z}^{*}$ l'inclusion naturelle envoyant $x$ sur $\langle x\rangle=(1,\cdots,x,1,\cdots)$, o\`u $x$ est \`a la place $p$. D'après \cite[Proposition 2.14]{Wang}, si $u\in \Z_p^*$, il existe une distribution 
\[z_{u,\mathrm{Eis,dR}}(k)= (u^2-\langle u\rangle)z_{\mathrm{Eis,dR}}(k)\] caractérisée par le fait que, quels que soient $r\in\Q^{*}$ et $(a,b)\in\Q^2$, on a
\[  \int_{(a+r\hat{\Z})\times(b+r\hat{\Z})}z_{u,\mathrm{Eis,dR}}(k)=\frac{1}{(k-2)!}r^{k-2}F^{(k)}_{u,r^{-1}a,r^{-1}b},    \]
où $F^{(k)}_{u,\alpha,\beta}= u^2F_{\alpha,\beta}^{(k)}-u^{2-k}F^{(k)}_{<u>\alpha,<u>\beta}$.
\begin{remark} Bannai-Kings \cite[Proposition 3.8]{BK} ont déterminé la réalisation de de Rham de la classe d'Eisenstein (cf. l'introduction). Ils la notent $\mathrm{Eis}_{\dR}^{k}(\varphi)$ avec $\varphi\in \cC^0((\Z/N)^2, \Q)$. Si $\phi$ est la transformée de Fourier de $\varphi$, vue comme une fonction localement constante 
sur $\A_f^2$, un calcul direct en passant aux $q$-développements, montre que\footnote{La normalisation vient de la relation de distribution.},  
\[\frac{1}{2}N^{k-2}\mathrm{Eis}_{\dR}^{k}(\varphi)=\int \phi z_{\mathrm{Eis,dR}}(k).\] 
Par ailleurs, R. Scheider \cite{RS} a obtenu une description de la réalisation de de Rham du polylogarithme elliptique et en déduit une description explicite de $\mathrm{Eis}_{\dR}^{k}(\varphi)$ compatible avec les autres.
\end{remark}

%% file: Siegel.tex
\section{Classe d'Eisenstein $p$-adique}\label{section2}
La realisation $p$-adique des classes d'Eisenstein, appelées classes d'Eisenstein $p$-adiques, a été déterminée par Kings \cite[theorem 4.2.9]{GK} par une construction géométrique. On donne une description purement algébrique des classes d'Eisenstein $p$-adiques, qui est compatible avec celle de Kings (cf. remarque \ref{compKings}).  

\subsection{Unités de Siegel}

Soit $K$ un sous corps de $\C$. On note $\cU(\Gamma, K)$ le groupe des unit\'es modulaires pour $\Gamma$ dont le $q$-d\'eveloppement est \`a coefficients dans $K$.
 On note $\cU(K)$ (resp.
$\cU^{\con}(K)$) la r\'eunion des $\cU(\Gamma,K)$, o\`u $\Gamma$ d\'ecrit tous les
sous-groupes d'indice fini (resp. de congruence) de $\SL_2(\Z)$.


Si $(\alpha,\beta)\in(\Q/\Z)^2$, $(c,6)=1$ et $(c\alpha,c\beta)\neq(0,0)$, on note $g_{c,\alpha,\beta}$ l'unité de Siegel définie par la formule:
\[g_{c,\alpha,\beta}= \frac{\theta^{c^2}(\tau, \tilde{\alpha}\tau+\tilde{\beta})}{\theta(\tau, c\tilde{\alpha}\tau+c\tilde{\beta})}, \text{ avec } (\tilde{\alpha},\tilde{\beta})\in \Q^2 \text{ qui a pour image } (\alpha,\beta)\in (\Q/\Z)^2.\] Elle ne dépend pas du choix de $(\tilde{\alpha},\tilde{\beta})\in \Q^2$ et appartient à $\cU^{\con}(\Q^{\cycl})$.
On note  $g_{\alpha,\beta}=g_{c,\alpha,\beta}^{1/(c^2-1)}\in \Q\otimes\cU(\overline{\Q})$, qui ne d\'epend pas du choix de $c\equiv1\mod N$. De plus, pour tout $c$, on a
$g_{c,\alpha,\beta}=g_{\alpha,\beta}^{c^2}g_{c\alpha,c\beta}^{-1}.$



 L'action de $\GL_2(\hat{\Z})$ sur $\cM^{\con}(\Q^{\cycl})$ induit une action de $\GL_2(\hat{\Z})$ sur $\Q\otimes \cU^{\con}(\Q^{\cycl})$. Soit $(\alpha,\beta)\in (\Q/\Z)^2$ et soit $\gamma\in \GL_2(\hat{\Z})$. Alors on a $g_{\alpha,\beta}*\gamma=g_{(\alpha,\beta)*\gamma}$, où  $(\alpha,\beta)*\gamma=(a\alpha+c\beta, b\alpha+d\beta)$ est le produit de matrices usuel avec $\gamma=(\begin{smallmatrix}a& b\\ c& d\end{smallmatrix})$.

Les relations de distribution des unités de Siegel se traduisent en l'énoncé suivant:
\begin{prop}\cite[Théorème 2.21]{Wang}Il existe une distribution alg\'ebrique 
\[z_{\mathrm{Siegel}}\in\fD_{\alg}(\A_f^2-(0,0),\Q\otimes\cU(\Q^{\cycl})),\]
telle que, quels que soient $r\in\Q^*$ et $(a,b)\in\Q^2-(r\Z,r\Z)$, on
ait:
\[\int_{(a+r\hat{\Z})\times(b+r\hat{\Z})}z_{\mathrm{Siegel}}=g_{r^{-1}a,r^{-1}b}(\tau).\]
De plus, $z_{\mathrm{Siegel}}$ est invariante sous l'action de
$\Pi_{\Q}$.
\end{prop}

\subsection{Th\'eorie de Kummer $p$-adique}
Soit $G$ un groupe localement profini. Soit $X$ un espace topologique
localement profini muni d'une action continue de $G$ \`a droite. Soit $M$ un $G$-module topologique muni d'une action \`a droite de
$G$. On note $\rH^{i}(G,M)$ le $i$-i\`eme groupe de cohomologie
continue de $G$ \`a valeurs dans $M$. 

Notons $Z^{0}=\{(x_n)_{n\in\N}|
x_n\in\cU(\overline{\Q}),(x_{n+1})^p=x_n\}$.  La topologie sur $\cU(\ol{\Q})$ est discr\`ete et on munit $Z^0$ de la topologie de la limite projective. 
Notons $Z=\Q\otimes Z^0$; $Z$ est muni d'une action de
$\Pi_{\Q}$ composante par composante.

On d\'efinit une projection  $\theta$ de $Z^0$ sur $\cU(\ol{\Q})$ en envoyant $(x_n)_{n\in\N}$ sur $x_0$.
 La projection $\theta:
Z^0\ra \cU(\ol{\Q})$ est surjective, et son noyau est
\[\ker(\theta)=\{(1,\z_p,\z_{p^{n}},...\z_{p^{n}},...)\}\cong\Z_p(1).\]
Autrement dit, on a la suite exacte de $\Pi_{\Q}$-modules topologiques:
\[0\ra\Z_p(1)\ra Z^0\ra\cU(\overline{\Q})\ra 0.\]

Dans la suite, on pose $X=\A_f^2-(0,0)$. Comme $\Q$ est plat sur $\Z$, on obtient la suite exacte de
$\Pi_{\Q}$-modules topologiques:
\[0\ra\fD_{\alg}(X,\Q_p(1))\ra \fD_{\alg}(X, Z\otimes\Q)\ra\fD_{\alg}(X,\cU(\overline{\Q})\otimes_\Z\Q)\ra 0.\]
En prenant la cohomologie continue de $\Pi_{\Q}$, on obtient une application de connexion "de Kummer":
\begin{equation*}
\rH^{0}(\Pi_{\Q},\fD_{\alg}(X,\cU(\overline{\Q})\otimes_\Z\Q))
\xrightarrow{\delta}\rH^{1}(\Pi_{\Q},\fD_{\alg}(X,\Q_p(1))).
\end{equation*}
On note $z_{\mathrm{Siegel}}^{(p)}\in \rH^{1}(\Pi_{\Q},\fD_{\alg}(X,\Q_p(1)))$ l'image de $z_{\mathrm{Siegel}}$ sous l'application de Kummer.  

\subsection{Torsion à la Soulé}\label{gcd}
Soit $G$ un groupe localement profini, agissant continûment à droite sur $X$. Soit $V$ une $\Q_p$-représentation de $G$ à droite. On note $\cC_c^0(X, V)$ le $\Q_p$-espace des fonctions continues à support compact sur $X$ à valeurs dans $V$ et $\fD_0(X,V)$ le $\Q_p$-espace des mesures sur $X$ à valeurs dans $V$.  On munit $\cC_c^0(X,V)$ et $\fD_0(X,V)$ d'actions de $G$ à droite comme suit:
si $g\in G$, $x\in X$,  $\phi(x)\in\cC_c^0(X,V)$, et $\mu\in\fD_0(X,V)$, alors 
\[ \phi*g(x)= \phi(x*g^{-1})*g \text{ et } \int_X \phi(x)(\mu*g)=\left(\int_X(\phi*g^{-1}) \mu\right)*g. \]
La proposition suivant est la "Torsion à la Soulé", utilisée aussi dans la construction de système d'Euler de Kato.
\begin{prop}\cite[proposition 2.17]{Wang1} Si $f\in\cC_c^0(X,V)^G $, alors la multiplication d'une mesure $\mu\in\fD_0(X,\Z_p)$ par la fonction $f$ induit un morphisme $G$-équivariant à droite de $\fD_0(X,\Z_p)$ dans $\fD_0(X, V)$. 
\end{prop}

D'après \cite[lemme 2.24]{Wang}, si $u\in \Z_p^*$, il existe un opérateur $r_u=u^2-\langle u\rangle$ tel que la distribution algébrique $z_{u,\mathrm{Siegel}}^{(p)}=r_uz_{\mathrm{Siegel}}^{(p)}$ appartienne à $\rH^1(\Pi_\Q, \fD_{\alg}(X,\Z_p(1)))$, et donc s'étende par continuité en une mesure. Ceci nous permet d'utiliser la "torsion à la Soulé".

On note $V_p=\Q_pe_1\oplus \Q_p e_2$ la représentation standard de $\GL_2(\Z_p)$ à droite donnée par les formules:
\[\text{ si } \gamma=(\begin{smallmatrix}a&b\\ c&d \end{smallmatrix}),  e_1*\gamma=ae_1+be_2 \text{ et } e_2*\gamma= ce_1+de_2. \]
On note $W_k=\Sym^{k-2}V_p$. Alors, la multiplication par la fonction $x\mapsto \frac{1}{(k-2)!}(ae_1+be_2)^{k-2}$, où $x_p=(a,b)$ est la composante à la place $p$ de $x$, induit un morphisme naturel:
\[\rH^1(\Pi_\Q, \fD_0(X,\Z_p(1)) )\ra \rH^1(\Pi_\Q, \fD_0(X, W_k(1))).\]
 On note $X^{(p)}=(\A_f^{]p[})^2\times (\Z_p^2-p\Z_p^2)$.
On note $z_{u,\mathrm{Eis},\text{ét}}(k)\in \rH^1(\Pi_{\Q},\fD_{0}(X^{(p)},W_k(1)))$ la restriction à $X^{(p)}$ de la mesure
\[(x\mapsto \frac{1}{(k-2)!}(ae_1+be_2)^{k-2})\otimes z^{(p)}_{u,\mathrm{Siegel}}.\]  On appelle $z_{u,\mathrm{Eis},\text{ét}}(k)$ la classe d'Eisenstein $p$-adique.
 
\begin{remark}\label{compKings}On vérifie facilement, en revenant aux définitions, que notre construction redonne celle de Kings \cite[theorem 4.2.9]{GK}. 
Plus précisément, si $(N,p)=1$, on a 
\begin{equation*}
\begin{split}
&\int_{(\alpha+N\hat{\Z})\times(\beta+N\hat{\Z})}z_{u,\mathrm{Eis},\text{ét}}(k)\\
=&\lim_{n\ra\infty} \sum_{\substack{(a,b)\equiv (\alpha,\beta)\mod N\\ 1\leq a,b\leq Np^n}}(ae_1+be_2)^{k-2}\delta(g_{u, a/Np^n, b/Np^n})\\
=&N^{k-2}\plim_n \sum_{\substack{(a,b)\equiv (\alpha,\beta)\mod N\\ 1\leq a,b\leq Np^n}}(a/Np^ne_1+b/Np^ne_2)^{k-2}\delta(g_{u, a/Np^n, b/Np^n}),
\end{split}
\end{equation*} 
où la limite projective $\plim_n \sum_{(a,b)\equiv (\alpha,\beta)\mod N}(a/Np^ne_1+b/Np^ne_2)^{k-2}\delta(g_{u, a/Np^n, b/Np^n})$ est la classe d'Eisenstein $p$-adique de Kings.
\end{remark}

%% file: loi.tex
\section{Une loi de r\'eciprocit\'e explicite }

\subsection{La méthode de Tate-Sen-Colmez}
Dans ce paragraphe, on rappelle les résultats de $\cite[\S 3, \S 4]{Wang}$ sur la méthode de Tate-Sen-Colmez pour l'anneau $\fK^+=\Q_p\{q/p\}$ des fonctions analytiques sur la boule $v_p(q)\geq 1$ \`a coefficients dans $\Q_p$. 
\subsubsection*{L'anneau $\fK^+$ et ses extensions }L'anneau $\fK^+$ est un anneau principal, complet pour la valuation $v_{p}$ d\'efinie par la formule:
\[ v_{p}(f)=\inf_{n\in\N}v_p(a_n), \text{ si } f=\sum_{n\in\N}a_n(q/p)^n\in \fK^+ .\]
On a $v_p(fg)=v_p(f)+v_p(g)$, ce qui permet de prolonger $v_p$ au corps des fractions de $\fK^+$ et on note $\fK$ son compl\'et\'e. Fixons une cl\^oture alg\'ebrique $\overline{\fK}$ de $\fK$ munie de la valuation $v_p$ qui est le prolongement unique de $v_p$ sur $\fK$ à $\overline{\fK}$. On note $\cG_{\fK}$ le groupe de Galois de $\overline{\fK}$ sur $\fK$.

Soit $M\geq 1$ un entier. On note $q_M$ (resp. $\z_M$) la racine
$M$-i\`eme $q^{1/M}$ (resp. $\exp(\frac{2i\pi}{M})$) de $q$ (resp.
$1$). On note $F_M=\Q_p[\z_M]$. Soit $\fK_M=\fK[q_{M},\z_M]$ ; c'est une extension galoisienne de $\fK$ de groupe de Galois $(\begin{smallmatrix}1 ,& \Z/M\Z\\ 0, &(\Z/M\Z)^*\end{smallmatrix})$.

On note $\fK_{\infty}$ (resp. $F_\infty$) la r\'eunion des $\fK_M$ (resp. $F_M$) pour tous $M\geq 1$. On note $P_{\Q_p}$ (resp. $P_{\ol{\Q}_p}$) le groupe de Galois de $\ol{\Q}_p\fK_{\infty}$ sur $\fK$ (resp. $\ol{\Q}_p\fK$) . Le groupe $P_{\ol{\Q}_p}$ est un groupe profini qui est isomorphe au groupe $\hat{\Z}$. De plus, on a une suite exacte:
\[1\ra P_{\ol{\Q}_p}\ra P_{\Q_p}\ra \cG_{\Q_p}\ra 1.\]

Fixons $M$ un entier $\geq 1$. On note $\fK_{Mp^{\infty}}$ (resp. $F_{Mp^{\infty}}$) la r\'eunion des $\fK_{Mp^n}$ (resp. $F_{Mp^\infty}$) pour tous $n\geq 1$.
Soit $\ol{\fK}^{+}$ la cl\^oture int\'egrale  de $\fK^{+}$ dans $\ol{\fK}$. On note $\fK_M^+$ la cl\^oture int\'egrale de $\fK^+$ dans $\fK_M$ et $\fK_\infty^+=\cup_M\fK_M^+$.

En associant son $q$-d\'eveloppement \`a une forme modulaire, on voit les formes modulaires comme des \'el\'ements de $\ol{\Q}_p\fK_{\infty}^+$.
Le groupe de Galois $P_{\Q_p}$ de $\bar{\Q}_p\fK_{\infty}$ sur $\fK$  pr\'eserve l'alg\`ebre des formes modulaires $\cM(\ol{\Q})$; autrement dit, $P_{\Q_p}$ est un sous-groupe de $\Pi_{\Q}$.


\subsubsection*{Les anneaux de Fontaine}
Soit $L$ un anneau de caract\'eristique $0$ muni d'une
valuation $v_p$ telle que
 $v_{p}(p)=1.$ On note $\cO_L=\{x\in L, v_p(x)\geq 0\}$ l'anneau
des entiers de $L$ pour la topologie $p$-adique. On note $\cO_{\C(L)}$ le compl\'et\'e  de $\cO_{L}$ pour la valuation $v_p$. On pose $\C(L)=\cO_{\C(L)}[\frac{1}{p}]$.

Soit $\R(L)=\plim A_n=\{(x_n)_{n\in \N}|
x_n\in \cO_L/p\cO_L \text{ et } x_{n+1}^p=x_n, \text{ si }n\in \N \}$.
Si $x=(x_n)_{n\in\N}\in\R(L)$, soit $\hat{x}_n$ un rel\`evement de
$x_n$ dans $\cO_{\C(L)}$. La suite $(\hat{x}_{n+k}^{p^k})$ converge quand
$k$ tend vers l'infini. Sa limite $x^{(n)}$ ne
d\'epend pas du choix des rel\`evements $\hat{x}_n$. On obtient ainsi
une bijection: $\R(L)\ra\{(x^{(n)})_{n\in\N}| x^{(n)}\in\cO_{\C(L)},
(x^{(n+1)})^p=x^{(n)}, \forall n\}$. 

L'anneau $\R(L)$ est un anneau parfait  de caract\'eristique $p$.
On note $\A_{\inf}(L)$ l'anneau des vecteurs de Witt \`a coefficients dans
$\R(L)$. Si $x\in \R(L)$, on note
$[x]=(x,0,0,...)\in \A_{\Inf}(L)$ son repr\'esentant de Teichm\"uller.
Alors tout \'el\'ement $a$ de $\A_{\inf}(L)$ peut s'\'ecrire de
mani\`ere unique sous la forme $\sum\limits_{k=0}^{\infty}p^k[x_k]$
avec une suite $(x_k)\in (\R(L))^{\N}$.

On d\'efinit un morphisme d'anneaux $\theta: \A_{\inf}(L)\ra\cO_{\C(L)}$
par la formule
$\sum_{k=0}^{+\infty}p^k[x_k]\mapsto
\sum_{k=0}^{+\infty}p^kx_k^{(0)}.$
 On note
$\B_{\inf}(L)=\A_{\inf}(L)[\frac{1}{p}]$,  et on \'etend $\theta$ en
un morphisme $\B_{\inf}(L)\ra \C(L).$
 On note $\B_m(L)=\B_{\inf}(L)/(\ker\theta)^m$. On fait de $\B_m(L)$ un anneau de Banach en prenant l'image de $\A_{\inf}(L)$ comme anneau d'entiers.

On d\'efinit $\B_{\dR}^{+}(L):=\plim \B_m(L)$ comme le compl\'et\'e
$\ker (\theta)$-adique de $\B_{\Inf}(L)$; on le munit de la topologie de la limite projective, ce qui en fait un anneau de Fr\'echet. Alors $\theta$ s'\'etend en
un morphisme continu d'anneaux topologiques
$\B_{\dR}^{+}(L)\ra \C(L)$. 

Pour simplifier la notation, on note $\A_{\inf}$ (resp. $\B_{\inf}$ et $\B_{\dR}^+$) l'anneau $\A_{\inf}(\ol{\fK}^+)$ (resp. $\B_{\inf}(\ol{\fK}^+)$ et $\B_{\dR}^+(\ol{\fK}^+)$).
Soit $\tilde{q}$ (resp. $\tilde{q}_M$ si $M\geq 1$ est un entier) le repr\'esentant de Teichm\"uller dans $\A_{\Inf}$ de $(q,
q_p,\cdots,q_{p^n},\cdots)$ (resp. $(q_M,\cdots,q_{Mp^n},\cdots)$). Si $M|N$, on a
$\tilde{q}_N^{N/M}=\tilde{q}_M$.

On d\'efinit une application continue $\iota_{\dR}:\fK^+\ra \B_{\dR}^+$ par $f(q)\mapsto f(\tilde{q}) $; 
ce qui permet d'identifier $\fK^+$ \`a un sous-anneau de $\B_{\dR}^+$. Mais il faut faire attention au fait que $\iota_{\dR}(\fK^+)$ n'est pas stable par $\cG_{\fK}$ car $\tilde{q}\sigma=\tilde{q}\tilde{\z}^{c_q(\sigma)}$ si
$\sigma\in\cG_{\fK}$, o\`u $c_{q}$ est le $1$-cocycle \`a valeur dans $\Z_p(1)$ associ\'e \`a $q$
par la th\'eorie de Kummer.

Posons
$\tilde{\fK}^+=\iota_{\dR}(\fK^+)[[t]]$.  Si $M\geq 1$ est un entier, on note $\tilde{\fK}^+_M$ l'anneau
$\tilde{\fK}^+[\tilde{q}_M,\tilde{\z}_M]$ et pose $\tilde{\fK}_{Mp^{\infty}}^+=\bigcup_{n}\tilde{\fK}^+_{Mp^n}$. L'application $\iota_{\dR}$ s'\'etend en un morphisme continu 
de $\fK^+$-modules $\iota_{\dR}:\fK^+_M\ra \B_{\dR}^{+}$ en envoyant $\z_M$ et $q_M$ sur $\tilde{\z}_M$ et $\tilde{q}_M$ respectivement. 
On a $\tilde{\fK}_M^+=\iota_{\dR}(\fK_M^+)[[t]]$.

On définit une application $\tilde{\fK}^+_M$-lin\'eaire de $\tilde{\fK}_{Mp^{\infty}}^+$ dans  $\tilde{\fK}_M^+$ par la formule:
\begin{equation*}
\begin{split}
\bR_M: \tilde{\fK}^+_{Mp^{\infty}}&\longrightarrow\tilde{\fK}^+_M \\
\tilde{\z}_{Mp^n}^{a}\tilde{q}_{Mp^n}^b &\mapsto
\left\{\begin{aligned}\tilde{\z}_{Mp^n}^{a}\tilde{q}_{Mp^n}^b& \empty ,
\text{ si } p^n|a \text{ et
} p^n|b;\\
0&   \empty, \text{ sinon}.
\end{aligned}\right.
\end{split}
\end{equation*}

\begin{prop}\cite[théorème 3.17, théorème 3.22]{Wang} Soit $M\geq 1$. On a \\
$(1)$
$\rH^0(\cG_{\fK_{Mp^{\infty}}},\B_{\dR}^{+})=\B_{\dR}^{+}(\fK_{Mp^{\infty}}^+)$;\\
$(2)$ $\tilde{\fK}_{Mp^{\infty}}^+$ est dense dans
$\B_{\dR}^{+}(\fK_{Mp^{\infty}}^+)$;\\
$(3)$ Si $M\geq 1$ est un entier tel que $m=v_p(M)\geq v_p(2p)$, l'application $\bR_M: \tilde{\fK}_{Mp^{\infty}}^+\ra \tilde{\fK}_M^+$ s'étend par continuité en une application $\tilde{\fK}^+_M$-lin\'eaire $\bR_M:\B_{\dR}^{+}(\fK_{Mp^{\infty}}^+)\ra\tilde{\fK}_M^+$. De plus, $\bR_M$ commute \`a l'action de $\cG_{\fK}$.
\end{prop}

\begin{prop}\cite[proposition 4.19]{Wang}\label{Tatenorm} Soit $v_p(M)\geq v_p(2p)$; si $V$ est une
$\Q_p$-repr\'esentation de $P_{\fK_M}$ munie d'un $\Z_p$-r\'eseau
$T$ tel que $P_{\fK_{M}}$ agit trivialement sur $T/2pT$, alors pour
$i\in\N$, $\bR_M$ induit un isomorphisme:
\[\bR_M:\rH^{i}(P_{\fK_M},\B_{\dR}^{+}(\fK^+_{Mp^{\infty}})\otimes V)\cong\rH^{i}(P_{\fK_M},\tilde{\fK}^+_M\otimes V).\]
\end{prop}

\subsection{Cohomologie des repr\'esentations du groupe $\rP_{m}$}\label{tech}

Soit $M\geq 1$ tel que $v_p(M)=m\geq v_p(2p)$.
Le groupe de Galois $P_{\fK_M}$ de l'extension
$\fK_{Mp^{\infty}}/\fK_M$ est un groupe analytique $p$-adique compact de rang $2$,
isomorphe \`a
\[\rP_m=\{(\begin{smallmatrix}a& b\\ c& d\end{smallmatrix})\in\GL_2(\Z_p): a=1, c=0, b\in p^m\Z_p, d\in 1+p^m\Z_p \},\]
et si $u,v\in p^m\Z_p$, on note $(u,v)$ l'élément $(\begin{smallmatrix}1&u\\ 0& e^v\end{smallmatrix})$ de $\rP_m$. La loi de groupe s'écrit sous la forme 
\[(u_1,v_1)(u_2,v_2)=(e^{v_2}u_1+u_2, v_1+v_2).\]
Soient $U_m$ et $\Gamma_m$ les sous-groupes de $\rP_m$ topologiquement engendrés par $(p^m,0)$ et $(0,p^m)$ respectivement. Ces deux sous-groupes sont isomorphes à $\Z_p$. De plus,  $U_m$ est distingué dans $\rP_m$ et on a $\rP_m/U_m=\Gamma_m$.  Soit $V$ une $\Q_p$-représentation de $\rP_m$. La suite spectrale de Hochschild-Serre nous fournit une suite exacte:
\[0\ra \rH^1(\Gamma, V^{U_m})\ra \rH^1(\rP_m, V)\ra \rH^1(U_m, V)^{\Gamma_m}.\] 

\begin{lemma}On a un isomorphisme $\rH^1(U_m, V)^{\Gamma_m}\cong (V/(u_m-1))^{e^{-p^m}\gamma_m=1}$.
\end{lemma}
\begin{proof} C'est un cas particulier de \cite[proposition 1.7.7]{NSW}.
\end{proof}
Soit $\rP$ un groupe analytique $p$-adique. On dira que l'action de $\rP$ sur une $\Q_p$-représentation de dimension finie $V$ est analytique si pour tout $\gamma\in\rP$ et $v\in V$, la fonction $x\mapsto \gamma^xv=\sum_{n=0}^{+\infty}\binom{x}{n}(\gamma-1)^nv$ est une fonction analytique sur $\Z_p$ à valeurs dans $V$.
On dira qu'une $\Q_p$-représentation de dimension finie $V$ de $\rP_m$ est analytique si l'action de $\rP_m$ est analytique sur $V$. 

Soit $V$ une $\Q_p$-représentation analytique de $\rP_m$.
Si $\gamma\in\rP_m$, on peut d\'efinir une d\'erivation $\partial_{\gamma}:V\ra V$ par rapport \`a $\alpha_{\gamma}$ par la formule:
\[\partial_{\gamma}(x)=\lim_{n\ra\infty}\frac{x*\gamma^{p^n}-x}{p^{n}} .\]
En particulier, on note $\partial_{m,i}$, $i=1,2$, les dérivations par rapport à $(p^m,0)$ et $(0,p^m)$ respectivement.
\begin{lemma}\label{sio}Soit $V$ une représentation analytique de $U_{m}$ munie d'un $\Z_p$-réseau $T$ stable sous l'action de $U_{m}$. Supposons que $(u_m-1)T\subset p^2T$. Alors, on a un isomorprhisme 
\[ \rH^1(U_m,T)\cong T/\partial_{m,1}.\]
\end{lemma}
\begin{proof}Ce lemme se démontre de la même manière que \cite[lemme 4.8]{Wang}.
\end{proof}
\begin{prop} \label{analytic}Soit $V$ une représentation analytique de $U_{m}$ munie d'un $\Z_p$-réseau $T$ stable sous l'action de $U_{m}$. Alors,\\
$(1)$ tout \'el\'ement de $ \rH^1(U_{m},T) $ est repr\'esentable par un $1$-cocycle analytique à un élément de $p^{m}$-torsion près;\\
$(2)$ l'image d'un $1$-cocycle analytique,
\[ u\mapsto c_{u}=\sum_{i\geq 1}c_i u^i,\]
sous l'isomorphisme $\rH^1(U_{m},T)\cong T/\partial_{m,1}$ est aussi celle de
$\delta^{(1)}(c_{u})=c_1$ dans
$T/\partial_{m,1}$  à un élément de $p^{m}$-torsion près.
\end{prop}
\begin{proof}Cette proposition se démontre de la même manière que \cite[proposition 4.5]{Wang1} (voir aussi \cite[théorème 4.1.10]{Wang}), mais plus simplement.
Le point clé est d'utiliser la relation de $1$-cocycle pour montrer que l'application
\[\delta^{(1)}: \{ 1\text{-cocycle analytique}\}\ra T\] induit une surjection 
$\rH^{1,\text{an}}(U_m,T)\ra T/(\partial_{m,1})$ à un $p^m$-torsion près.

\end{proof}


\subsection{Construction d'une application  exponentielle duale}\label{constructiondeloi}

On note $\fK^{++}=\Z_p\{q/p\}$ l'anneau des entiers de $\fK^{+}$ pour la valuation $v_p$, ainsi que 
$\cK^{++}=\Z_p[[q/p]]$ son compl\'et\'e $q/p$-adique. 
On note $\fK_M^{++}$ l'anneau des entiers de $\fK_M^+$, qui est l'anneau
\[\{ \sum_{n=0}^{+\infty}a_nq_M^n\in \fK_M^+: a_n\in F_M \text{ tel que } v_p(a_n)+\frac{n}{M} \geq 0, \text{ pour tout } n  \},\]
et  on note $\cK_M^{++}$ son compl\'et\'e $q/p$-adique, 
ainsi que $\cK_M^+=\cK_M^{++}\otimes \Q_p$.

Rappelons que l'application $\iota_{\dR}:\fK^+\ra \B_{\dR}^+; f(q)\mapsto f(\tilde{q})$ identifie $\fK^+$ \`a un sous-anneau de $\B_{\dR}^+$. On note 
$\tilde{\fK}^+=\iota_{\dR}(\fK^+)[[t]]$ et $\tilde{\cK}^+=(\widehat{\iota_{\dR}(\fK^{++})}\otimes\Q_p)[[t]]$, où $\widehat{\iota_{\dR}(\fK^{++})}$ est le complété $\tilde{q}/p$-adique de $\iota_{\dR}(\fK^{++})$. De même, on note $\tilde{\fK}_M^+=\iota_{\dR}(\fK_M^+)[[t]]$ et $\tilde{\cK}_{M}^+=(\widehat{\iota_{\dR}(\fK_M^{++})}\otimes\Q_p)[[t]]$. 
On a 
\[\widehat{\iota_{\dR}(\fK_M^{++})}=\{ \sum_{n=0}^{+\infty}a_{n}\tilde{q}_M^n\in F_M[[\tilde{q}_M]]: a_{n}\in F_M \text{ tel que } v_p(a_{n})+\frac{n}{M} \geq 0\}.\]
On d\'efinit une application $\theta: \tilde{\cK}^{+}_M\ra \cK^{+}_M$ par r\'eduction modulo $t$, qui est compatible avec celle définie sur $\tilde{\fK}_M^{+}$.
On constate que $\tilde{\fK}_M^+$ est la limite projective $\plim_n(\tilde{\fK}_M^+/t^n)$, où les $\tilde{\fK}_M^+/t^n$ sont des $\fK^+$-modules de rang fini munis de la topologie définie par $v_p$. 

Posons $\tilde{V}=\plim_{n} (\tilde{\fK}_M^+/t^{n})\hat{\otimes} W_k(1)$ la $\Q_p$-représentation de $\rP_m$, qui n'est pas une $\Q_p$-repr\'esentation analytique, mais qui peut s'approximer par les representations analytiques $\tilde{T}_{n_1,n_2}$ définis ci-dessous: 
On note $T(1)=\Sym^{k-2}(\Z_pe_1\oplus \Z_pe_2) (1)$ la structure entière de $W_k(1)$. On note $\tilde{T}_{n_1}=(\widehat{\iota_{\dR}(\fK_M^{++})}\hat{\otimes}T(1)[[t]])/t^{n_1}$. C'est un $\Z_p$-réseau de  $(\tilde{\fK}_M^+/t^{n_1})\hat{\otimes} W_k(1)$ stable sous l'action de $\rP_m$.  On note $\m_{n_1}^{n_2}$ le sous $\Z_p$-module de 
$\tilde{T}_{n_1}$ des éléments avec $(\tilde{q}/p)$-adique valuation $\geq n_2$.
 Comme $\m_{n_1}^{n_2}$ est stable sous l'action de $\rP_m$, on peut définir les $\Z_p$-représentations analytiques $\tilde{T}_{n_1,n_2}$ de $\rP_m$, pour tous $n_1,n_2\geq 1$,
\[\tilde{T}_{n_1,n_2}=(\widehat{\iota_{\dR}(\fK_M^{++})}\hat{\otimes}T(1)[[t]])/t^{n_1})/\m_{n_1}^{n_2}.\]
  
L'inclusion $ \tilde{V}\subset\plim_{n_1}\left((\plim_{n_2} \tilde{T}_{n_1,n_2})\otimes\Q_p\right)$ de $\rP_m$-représentations, nous permet de définir un morphisme:
\[\rH^i(\rP_m, \tilde{V})\ra \plim_{n_1}\rH^i(\rP_m, (\plim_{n_2} \tilde{T}_{n_1,n_2})\otimes\Q_p) \ra \plim_{n_1}\left((\plim_{n_2}\rH^i(\rP_m,  \tilde{T}_{n_1,n_2}))\otimes\Q_p\right).  \]
\begin{lemma}\cite[lemme 4.10]{Wang1} Les actions de $\partial_{m,1}$ et $\partial_{m,2}-p^m$ sur $t$ et $\tilde{q}_M$ sont donn\'ees par les formules suivantes:
\begin{equation*}
\begin{split}
\partial_{m,1}(t)=0, &\partial_{m,1}(\tilde{q}_M)=\frac{p^mt}{M}\tilde{q}_M;
\partial_{m,2}(t)=p^mt,\partial_{m,2}(\tilde{q}_M)=0.\end{split}
\end{equation*}

\end{lemma}

\begin{prop}\label{reskj}Si $v_p(M)=m\geq v_p(2p)$ , alors l'application 
\[f(q_M)\mapsto e_1^{k-2}  t f(\tilde{q}_M) \] induit un isomorphisme de $\cK_{M}^{+}$ sur $\plim_{n_1}\left((\plim_{n_2}(\tilde{T}_{n_1,n_2}/\partial_{m,1})^{\partial_{m,2}=p^m})\otimes \Q_p\right)$.
\end{prop}
\begin{proof}On note $\bM_{n_1,n_2}$ le sous-$\Z_p$-module de $\tilde{T}_{n_1,n_2}$ des éléments de la forme
\[ \sum\limits_{\substack{0\leq r\leq Mn_2-1;\\ 0\leq s\leq n_1-1}}a_{r,s}e_1^{k-2}\tilde{q}_M^rt^s  \text{ où } a_{r,s}\in F_M \text{ vérifie } v_p(a_{r,s})+\frac{r}{M}\geq 0 .\]
 On constate qu'il n'existe pas d'élément de $\tilde{T}_{n_1,n_2}$ tel que $\partial_{m,1} x$ appartienne à $\bM_{n_1,n_2}$. Par conséquent, l'application naturelle
\[\phi_1: \bM_{n_1,n_2}\ra \tilde{T}_{n_1,n_2}/\partial_{m,1} \] 
est injective.   
On vérifie facilement que le $\Z_p$-module $\bM_{n_1,n_2}+\partial_{m,1}\tilde{T}_{ n_1,n_2}$ contient $p^{m(n_1+1)}\tilde{T}_{ n_1,n_2}$, pour tout $n_1,n_2\geq 1$. Ceci implique que le conoyau $\coker\phi_1$ de $\phi_1$ est un $\Z_p$-module de $p^{m(n_1+1)}$-torsion.

D'autre part, comme on a  
\begin{equation}\label{partial2}(\partial_{m,2}-p^m)( e_1^{k-2} \tilde{q}_M^r t^s)=p^m(s-1)e_1^{k-2}\tilde{q}_M^rt^s,
\end{equation} 
le $\Z_p$-module $\bM_{n_1,n_2}$ est stable sous l'action de $\partial_{m,2}-p^m$.
Donc l'application $\phi_1$ induit une application injective, que l'on note encore par $\phi_1$,
\[\phi_1:\bM_{n_1,n_2}^{\partial_{m,2}=p^m}\ra( \tilde{T}_{n_1,n_2}/\partial_{m,1})^{\partial_{m,2}=p^m}.\]
De plus, son conoyau est un $\Z_p$-module de $p^{m(n_1+1)}$-torsion.

Si $n_2\geq 1$, on note $(\cK_{M}^{++})_{n_2}=(\cK_M^{++})/(q/p)^{n_2}$.
Pour tout $n_1\geq 1$, on dispose d'une application  $\phi_0: (\cK_M^{++})_{n_2}\ra \bM_{n_1,n_2}$ en envoyant $f(q_M)$ sur $f(\tilde{q}_M)e_1^{k-2}t$, qui est une injection. De la formule $(\ref{partial2})$ pour $s=1$, on déduit que $\phi_0$ induit une application injective, notée encore $\phi_0$,
 \[\phi_0:(\cK_M^{++})_{n_2}\ra\bM_{n_1,n_2}^{\partial_{m,2}=p^m}.\]
 En composant avec l'application $\phi_1$, on obtient une application injective 
  \[\phi=\phi_1\circ\phi_0: (\cK_M^{++})_{n_2}\ra (\tilde{T}_{n_1,n_2}/\partial_{m,1})^{\partial_{m,2}=p^m}, \]
  
 En prenant la limite projective sur $n_2$, on obtient une injection 
 \[\cK_M^{++} \ra \plim_{n_2}(\tilde{T}_{n_1,n_2}/\partial_{m,1})^{\partial_{m,2}=p^m} \text{ pour } n_1\geq 1 .\] 
 Il ne reste qu'à montrer la surjectivité de
  \[\cK^{+}_M\ra (\plim_{n_2}(\tilde{T}_{n_1,n_2}/\partial_{m,1})^{\partial_{m,2}=p^m})\otimes \Q_p.\] 
 Cela se ramène à montrer que les applications 
 \[\plim_{n_2}\phi_0: \cK^{+}_M\ra (\plim_{n_2}\bM_{n_1,n_2}^{\partial_{m,2}=p^m})\otimes\Q_p\] et 
 \[\plim_{n_2}\phi_1: (\plim_{n_2}\bM_{n_1,n_2}^{\partial_{m,2}=p^m})\otimes\Q_p\ra (\plim_{n_2}(\tilde{T}_{n_1,n_2}/\partial_{m,1})^{\partial_{m,2}=p^m})\otimes \Q_p \] 
 sont surjectives.  La surjectivité de $\plim_{n_2}\phi_0$ découle de la formule  $(\ref{partial2})$ et celle de $\plim_{n_2}\phi_1$ découle du fait que, pour tout $n_1\geq 1$, le conoyau de $\phi_1$ est de $p^{m(n_1+1)}$-torsion.  

\end{proof}

En composant les applications obtenues dans les paragraphes pr\'ec\'edents, on obtient le diagramme suivant:
\[
\xymatrix{
 \rH^1(  \cG_{\fK_{M}},\B_{\dR}^{+}\otimes  W_{k}(1)         ) &
\\
\rH^1(  P_{\fK_{M}},\B_{\dR}^{+}(\fK^+_{Mp^{\infty}})\otimes W_{k}(1)   ) \ar[u]^-{(1)}\ar[r]^-{(2)}\ar@{..>}[dd]^{\exp^{*}} &
\rH^1\left(  P_{\fK_{M}}, \tilde{V} \right)\ar[d]^-{(3)}
\\
 & \plim_{n_1}\left((\plim_{n_2}\rH^1( P_{\fK_{M}},\tilde{T}_{n_1,n_2}) ) \otimes_{\Z_p} \Q_p  \right) \ar[d]^-{(4)}
\\
\cK_{M}^+ &\plim_{n_1}\left((\plim_{n_2} (\tilde{T}_{n_1,n_2}/\partial_{m,1})^{\partial_{m,2}=p^m}\otimes _{\Z_p}\Q_p\right)\ar[l]^-{\cong}_-{(5)},
}
\]
o\`u\\
$\bullet$ l'application $(1)$, d'inflation, est injective car $(\B_{\dR}^{+})^{\cG_{\fK_{Mp^{\infty}}}}=\B_{\dR}^{+}(\fK_{Mp^{\infty}}^+)$ et $\cG_{\fK_{Mp^{\infty}}}$ agit trivialement sur $W_{k}(1)$;\\
$\bullet$ $(2)$ est l'isomorphisme induit par "la trace de Tate normalis\'ee" $\bR_M$ (cf. proposition $\ref{Tatenorm}$);\\
$\bullet$ $(3)$ est l'application naturelle induite par la projection ;\\
$\bullet$ $(4)$ est induite par l'application de restriction et l'isomorphisme du lemme $\ref{sio}$ ; \\
$\bullet$ $(5)$ est l'isomorphisme dans la proposition \ref{reskj}.\\

On d\'efinit l'application $\exp^{*}$  en composant les applications $(2),(3), (4), (5)$.
  \begin{recette}\label{res} Comme $\tilde{T}_{n_1,n_2}$ est une repr\'esentation analytique pour tout $n_1$ et $n_2$, l'application $(5)$ se calcule gr\^ace au théorème $\ref{analytic}$. Plus pr\'ecis\'ement, cela se fait comme suit:
 on d\'efinit une application $\res_{k}^{(n_1,n_2)}: \tilde{T}_{n_1}^{\partial_2=1}\ra \cK_{M}^+$ en composant la
projection\footnote{Comme $\partial_1 $ et $\partial_2$ se commutent, on peut prendre le sous-module fixé par $\partial_2$, ensuite prendre le quotient.} 
$\tilde{T}_{n_1}^{\partial_2=1}\ra (\tilde{T}_{n_1,n_2}/\partial_1)^{\partial_2=1}$ avec l'inverse de l'isomorphisme dans la proposition
$\ref{reskj}$. En prenant la limite projective sur $n_2$, on obtient un morphisme $\res_{k}^{(n_1)}:\tilde{V}_{k}^{\partial_2=1}\ra \cK_M^+$. Si $c=(c^{(n_1)})\in \plim \rH^1(U_{\fK_M},\tilde{T}_{n_1} )^{\partial_2=1}$ est repr\'esent\'e par une limite de $1$-cocycle analytique $\sigma\mapsto c^{(n_1)}_{\sigma}$ sur $U_{\fK_M}$ \`a valeurs dans $\tilde{T}_{n_1}^{\partial_2=1} $, alors l'image de $c$ dans $\cK_M^+$ est $\res_k(\delta^{(1)}(c))=\plim_{n_1}\res_{k}^{(n_1)}(\delta^{(1)}(c))$, o\`u $\delta^{(1)}$ est l'application d\'efinie dans la proposition $\ref{analytic}$.
\end{recette}
\subsection{Application à la classe d'Eisenstein $p$-adique}
\subsubsection{Énoncé du théorème principal}

Soit $M\geq 1$, $p|M$ et $A=(\alpha,\beta)$ avec $(\alpha, \beta)\in \{1,\cdots, M\}$ et $(\alpha,\beta)\notin p\Z^2$. On note $\psi_{M,A}$ la fonction caractéristique $1_{A+M\hat{\Z}^2}$. C'est une fonction invariante sous l'action de $\cG_{\fK_M}$. Par ailleurs, l'image de la classe d'Eisenstein $p$-adique $z_{u,\mathrm{Eis},\text{ét}}(k)$ sous l'application de localisation appartient à $\rH^1(\cG_{\fK_M}, \fD_{0}((\hat{\Z}^{(p)})^2, W_k(1)))$, et on a 
$\int\psi_{M,A}z_{u,\mathrm{Eis},\text{ét}}(k)\in \rH^1(\cG_{\fK_M}, W_k(1))$.
On note son image dans $\rH^1(\cG_{\fK_M}, \B_{\dR}^+(\fK^+_{Mp^{\infty}})\otimes W_k(1))$ par $z_{M,A}$. 

\begin{prop}\label{final}Pour toute paire $(M,A)$ ci-dessus et $v_p(M)\geq v_p(2p)$, on a 
\[ \exp^*(z_{M,A})= \frac{1}{(k-2)!}M^{k-2} F^{(k)}_{u,\frac{\alpha}{M},\frac{\beta}{M}}.\]
\end{prop}
On peut condenser cet énoncé en l'énoncé suivant (avec des notations évidentes):
\begin{theo}\label{principal} Si $k\geq 2$ et si $u\in \Z_p^*$, on a l'égalité suivante dans $\rH^0(\Pi_\Q, \fD_{\alg}(X^{(p)}, \fK_\infty^+))$,
\[\exp^*(z_{u,\mathrm{Eis},\text{ét}}(k))= z_{u,\mathrm{Eis, dR}}(k).\]
\end{theo}
\subsubsection{Construction d'un $1$-cocycle}

Soit $a,b \in \{1,\cdots p^nM\}$ vérifiant: $(a,b)\equiv (\alpha,\beta)\mod M$. On note les fonctions caractéristiques $1_{(a+Mp^n\Z_p)\times(b+Mp^n\Z_p)}$ par $\psi^{(n)}_{a,b}$. Notons $U$ l'ouvert $(\alpha+M\Z_p)\times(\beta+M\Z_p)$ de $\Z_p^2$. 
On définit une mesure $\mu\in \rH^1(\cG_{\fK_M},\fD_0(U,\Z_p(1)))$ par la formule:
\[ \int\psi^{(n)}_{a,b}\mu=\int_{(a+Mp^n\hat{\Z})\times(b+Mp^n\hat{\Z})}z_{u,\mathrm{Siegel}}^{(p)}.\]
On a $z_{u,\mathrm{Eis},\text{ét}}(k)= \frac{1}{(k-2)!}(a_pe_1+b_pe_2)^{k-2}\otimes z_{u,\mathrm{Siegel}}^{(p)}$. Ceci implique que
\[z_{M,A}=\int\psi_{M,A}z_{u,\text{Eis,ét}}(k)=\int_U\frac{1}{(k-2)!}(a_pe_1+b_pe_2)^{k-2}\mu.\]

Soit $\Psi$ une base du $\Z$-module des fonctions localement constantes sur $U$ constituée de fonctions du type $\psi_{a,b}^{(n)}$, avec $n\in \N$ et $(a,b)$ comme ci-dessus. On définit une distribution algébrique $\mu_{\Psi}$ sur $U$ à valeurs dans $\B_{\dR}^{+}(\ol{\fK}^+)[u_q]$ avec $u_q=\log \tilde{q}$ par la formule:
si $\psi_{a,b}^{(n)}\in \Psi$, 
\[\int \psi_{a,b}^{(n)}\mu_{\Psi}= \log g_u(\tilde{q}, \tilde{q}_{Mp^n}^a\tilde{\z}_{Mp^n}^b),\]
où l'élément $g_{u}(\tilde{q}, \tilde{q}_{Mp^n}^a\tilde{\z}_{Mp^n}^b)$ est obtenu en remplaçant les variables $q_{Mp^n}$ et $\zeta_{Mp^n}$ de la fonction $g_{u, \frac{a}{Mp^n}, \frac{b}{Mp^n}}$ par $\tilde{q}_{Mp^n}$ et $\tilde{\z}_{Mp^n}$ respectivement.

On identifie $\Z_p(1)$ au sous-module de $\B_{\dR}^+$ via l'isomorphisme de $\Z_p[\cG_{\fK}]$-modules:
\[\log\circ [\cdot]: \Z_p[1]\ra \Z_pt .\]
\begin{lemma}\cite[lemme 5.10]{Wang}
Considérons l'application \[\rH^1(\cG_{\fK_M}, \fD_0(U,  \Z_p(1) ))\ra \rH^1(\cG_{\fK_M}, \fD_0(U,  t\B_{\dR}^+ )).\] Alors $\mu$ est représenté par le $1$-cocycle
$\gamma\mapsto \mu_{\Psi}*(\gamma-1)$,
qui est l'inflation d'un $1$-cocycle sur $P_{\Q_p}^{\cycl}=\Gal(\fK_\infty/\fK)$ à valeurs dans $\fD_0(U,   t\B_{\dR}^+(\fK_{Mp^\infty}^+))$.\end{lemma}


\subsubsection{Descente de $\fK_{Mp^\infty}$ à $\fK_M$}
D'après ce qui précède, $z_{M,A}$ est la classe du $1$-cocycle 
\[\gamma\mapsto \int_U\frac{1}{(k-2)!}(a_pe_1+b_pe_2)^{k-2}(\mu_{\Psi}*(\gamma-1)),\]
qui est aussi la classe du $1$-cocycle par "la trace de Tate normalisée":
 \[\gamma\mapsto \bR_M\left(\int_U\frac{1}{(k-2)!}(a_pe_1+b_pe_2)^{k-2}(\mu_{\Psi}*(\gamma-1))\right).\]
 
 Le lemme suivant se démontre par un calcul facile, qui est un analogue de \cite[lemme 5.13]{Wang}. 
\begin{lemma}Si $(a,b)\notin p\Z_p^2$, alors on a 
\[\bR_M(\log(\theta(\tilde{q}, \tilde{q}_{Mp^n}^a\tilde{\z}_{Mp^n}^b)))= p^{-n}\log(\theta(\tilde{q}^{p^n}, \tilde{q}_{M}^a\tilde{\z}_{M}^b)).\]
\end{lemma}

\begin{prop} Si on note 
\[\log_{a,b}^{(n),\gamma}=\log((u^2-\langle u\rangle)\theta(\tilde{q}^{p^n}, \tilde{q}^a_M\tilde{\z}_M^b))*(\gamma-1),\]
alors $z_{M,A}$ est représentée par le $1$-cocycle
\[\gamma\mapsto\lim_{n\ra \infty} \frac{1}{(k-2)!p^n}\sum (ae_1+be_2)^{k-2}\log_{a,b}^{(n),\gamma} ,\]
la somme portant sur l'ensemble 
\[U^{(n)}=\{(a,b)\in \{1,\cdots Mp^n\}^2: a\equiv \alpha, b\equiv \beta\mod M\}.\]
\end{prop}
\begin{proof} $z_{M,A}$ est représentée par le $1$-cocycle
\[\gamma\mapsto \bR_M\left(\int_U\frac{1}{(k-2)!}(a_pe_1+b_pe_2)^{k-2}(\mu_{\Psi}*(\gamma-1))\right).\]
Par la définition de l'intégration sur $U$, on a 
\[\int_U\frac{1}{(k-2)!}(a_pe_1+b_pe_2)^{k-2}(\mu_{\Psi}*(\gamma-1))=\lim_{n\ra\infty} \sum_{(a,b)\in U^{(n)}} \frac{1}{(k-2)!}(ae_1+be_2)^{k-2}\int\psi_{a,b}^{(n)}(\mu_{\Psi}*(\gamma-1)).\]

Comme $\bR_M$ commute avec l'action de $P_{\fK_M}$, on a 
\begin{equation}
\begin{split}
&\bR_M\left(\int_U\frac{1}{(k-2)!}(a_pe_1+b_pe_2)^{k-2}(\mu_{\Psi}*(\gamma-1))\right)\\
&=\lim_{n\ra\infty} \sum_{(a,b)\in U^{(n)}} \frac{1}{(k-2)!}(ae_1+be_2)^{k-2}\bR_M\left(\int\psi_{a,b}^{(n)}(\mu_{\Psi}*(\gamma-1))\right)\\
&=\lim_{n\ra\infty} \sum_{(a,b)\in U^{(n)}} \frac{1}{(k-2)!}(ae_1+be_2)^{k-2} r_u\left(\bR_M( \log\theta(\tilde{q},\tilde{q}_{Mp^n}^{a}\tilde{\z}_{Mp^n}^b)* (\gamma-1) )   \right)\\
&=\lim_{n\ra\infty}\sum_{(a,b)\in U^{(n)}} \frac{1}{(k-2)!p^n}(ae_1+be_2)^{k-2}\log_{a,b}^{(n),\gamma}.
\end{split}
\end{equation}
\end{proof}
\subsubsection{Passage à l'algèbre de Lie}
Comme le $1$-cocycle $\gamma\mapsto\lim_{n\ra \infty} p^{-n}\sum \frac{1}{(k-2)!}(ae_1+be_2)^{k-2}\log_{a,b}^{(n),\gamma} $ est la limite de $1$-cocycles analytiques à valeurs dans $t\fK_{M}^+\otimes S_k$, on utilise les techniques différentielles pour calculer son image dans $\cK_M^+$.

Si $f(x_1,x_2)$ est une fonction en deux variables, on note $D_2$ l'opérateur $x_2\frac{d}{dx_2}$. Si $n\in \N$ et $a,b\in \Z$, on pose $f_{a,b}^{(n)}=f(\tilde{q}^{p^n}, \tilde{q}_M^a\tilde{\z}_M^b)$.
\begin{lemma}On note $\delta_{a,b}^{(1)}=\delta^{(1)}(\log_{a,b}^{(n),\gamma})$. On a 
\[\delta_{a,b}^{(1)}=\frac{at}{M}D_2\log(r_u\theta_{a,b}^{(n)}).\]
\end{lemma}
\begin{proof}Soit $f_{a,b}^{(n)}$ définie comme ci-dessus. Alors l'action de $(u,v)\in \rP_m$ sur $f_{a,b}^{(n)}$ est 
\[(u,v)f_{a,b}^{(n)}= f(\tilde{q}^{p^n}, \tilde{q}_M^a\tilde{\z}_M^b)+u\frac{at}{M}D_2f_{a,b}^{(n)}+v\frac{bt}{M}D_2f_{a,b}^{(n)}+O((u,v)^2);\] 
donc l'action de $(\begin{smallmatrix} 1& u\\ 0 &e^v\end{smallmatrix})-1$ sur $f_{a,b}^{(n)}$ est donnée par la formule:
\begin{equation}
f_{a,b}^{(n)}*((\begin{smallmatrix} 1& u\\ 0 &e^v\end{smallmatrix})-1)= \frac{au+bv}{M}tD_2f_{a,b}^{(n)}+O((u,v)^2).
\end{equation}
On déduit de la définition de l'application $\delta^{(1)}$ que $\delta^{(1)}_{a,b}=\frac{at}{M}D_2\log(r_u\theta_{a,b}^{(n)})$.
\end{proof}
Le corollaire suivant nous permet de terminer la démonstration de la proposition \ref{final}.
\begin{coro}On a 
\begin{equation}
\begin{split}\mathrm{res}_k\left(\lim\limits_{n\ra \infty} \frac{1}{(k-2)!p^{n}}\sum (ae_1+be_2)^{k-2}\delta_{a,b}^{(1)}  \right)=\frac{1}{(k-2)!}M^{k-2} F_{u,\frac{\alpha}{M},\frac{\beta}{M}}^{(k)}.
\end{split}
\end{equation}
\end{coro}
\begin{proof}Par la définition de l'application $\mathrm{res}_k$ (cf. recette \ref{res}), on a 
\[\mathrm{res}_k\left(\lim\limits_{n\ra \infty} \frac{1}{(k-2)!p^{n}}\sum_{\substack{a\equiv\alpha [M]\\
b\equiv\beta[M]\\ 1\leq a,b\leq Mp^n}} (ae_1+be_2)^{k-2}\delta_{a,b}^{(1)}  \right)= \lim\limits_{n\ra \infty} \sum\frac{1}{(k-2)!p^{n}}\frac{a^{k-1}}{M}D_2\log(r_u\theta_{a,b}^{(n)}).\]
Comme $\partial_zE_{u,\alpha,\beta}^{(k)}=E_{u,\alpha,\beta}^{k+1}$ pour $k\geq 0$, on a 
\[D_2\log(r_u\theta_{a,b}^{(n)})= E_{u,1}(q^{p^n}, q_M^a\z_M^b)=E_{u,1}(q^{p^n}, q_M^a\z_M^\beta).\]
Ceci implique que
\[\mathrm{res}_k\left(\lim\limits_{n\ra \infty} \frac{1}{(k-2)!p^{n}}\sum_{\substack{a\equiv\alpha [M]\\
b\equiv\beta[M]\\ 1\leq a,b\leq Mp^n}} (ae_1+be_2)^{k-2}\delta_{a,b}^{(1)}  \right)=\lim\limits_{n\ra \infty} \sum_{\substack{a\equiv\alpha [M]\\
1\leq a\leq Mp^n}}\frac{1}{(k-2)!}\frac{a^{k-1}}{M}E_{u,1}(q^{p^n}, q_M^a\z_M^\beta).\]
Par ailleurs, on a la formule (cf. \cite[lemme 5.18]{Wang}) suivante:
\[ \lim\limits_{n\ra \infty} \sum_{\substack{a\equiv\alpha [M]\\ 1\leq a\leq Mp^n}} a^rE_{u,1}(q^{p^n}, q_M^{a}\z_M^\beta)=M^rF_{u,\alpha/M,\beta/M}^{(r+1)}.\]
Ceci permet de conclure le corollaire.

\end{proof}